\documentclass[leqno, 12pt]{article}
\usepackage{amsmath,amsfonts,amsthm,amssymb,indentfirst}
\usepackage{xy} \xyoption{all}
\usepackage[colorlinks]{hyperref}
\hypersetup{linkcolor=blue, urlcolor=blue, citecolor=red}

\newtheorem{theorem}{Theorem}
\newtheorem{lemma}[theorem]{Lemma}
\newtheorem{definition}[theorem]{Definition}
\newtheorem{corollary}[theorem]{Corollary}
\newtheorem{proposition}[theorem]{Proposition}

\newtheorem{question}[theorem]{Question}

\theoremstyle{definition}
\newtheorem{example}[theorem]{Example}

\newcommand{\End}{\mathrm{End}}

\newcommand{\kernel}{\mathrm{ker}}

\newcommand{\Z}{\mathbb{Z}}

\newcommand{\N}{\mathbb{N}}
\newcommand{\M}{\mathbb{M}}

\newcommand{\B}{\mathcal{B}}
\newcommand{\T}{\mathcal{T}}
\newcommand{\U}{\mathcal{U}}

\begin{document}
 
\title{Infinite-Dimensional Triangularization}
\author{Zachary Mesyan}

\maketitle

\begin{abstract}
The goal of this paper is to generalize the theory of triangularizing matrices to linear transformations of an arbitrary vector space, without placing any restrictions on the dimension of the space or on the base field. We define a transformation $T$ of a vector space $V$ to be \emph{triangularizable} if $V$ has a well-ordered basis such that $T$ sends each vector in that basis to the subspace spanned by basis vectors no greater than it. We then show that the following conditions (among others) are equivalent: (1) $T$ is triangularizable, (2) every finite-dimensional subspace of $V$ is annihilated by $f(T)$ for some polynomial $f$ that factors into linear terms, (3) there is a maximal well-ordered set of subspaces of $V$ that are invariant under $T$, (4) $T$ can be put into a crude version of the Jordan canonical form. We also show that any finite collection of commuting triangularizable transformations is simultaneously triangularizable, we describe the closure of the set of triangularizable transformations in the standard topology on the algebra of all transformations of $V$, and we extend to transformations that satisfy a polynomial the classical fact that the double-centralizer of a matrix is the algebra generated by that matrix.

\medskip


\end{abstract}

% 15A04 - Linear transformations
% 15A21 - Canonical forms, reductions, classification
% 16S50 - Endomorphism rings
% 16W80 - Topological and ordered rings and modules

\section{Introduction}

The following summarizes much of the existing wisdom on triangularizing a linear transformation of a finite-dimensional vector space. Our main goal is to generalize this to transformations of vector spaces of arbitrary dimension over an arbitrary field.

\begin{theorem}[Classical Triangularization Theorem] \label{cl-tri-thrm}
Let $k$ be a field, $V$ a finite-dimensional $k$-vector space, and $\, T$ a linear transformation of $\, V$. Then the following are equivalent.
\begin{enumerate}
\item[$(1)$] $T$ has an upper-triangular representation as a matrix with respect to some basis for $\, V$.
\item[$(1')$] $T$ has a lower-triangular representation as a matrix with respect to some basis for $\, V$.
\item[$(2)$] There is a polynomial $p(x) \in k[x]\setminus k$ that factors into linear terms in $k[x]$, such that $p(T)=0$.
\item[$(3)$] There exists a well-ordered set of $T$-invariant subspaces of $\, V$, which is maximal as a well-ordered set of subspaces of $V$.
\item[$(3')$] There exists a totally ordered set of $T$-invariant subspaces of $\, V$, which is maximal as a totally ordered set of subspaces of $V$.
\item[$(4)$] $T$ has a representation as a matrix in Jordan canonical form with respect to some basis for $V$.
\end{enumerate}
\end{theorem}

\begin{proof}
$(1) \Leftrightarrow (1')$ The transformation $T$, viewed as a matrix, is upper-triangular with respect to a basis $v_1, \dots, v_n$ for $V$ if and only if it is lower-triangular with respect to $v_n, \dots, v_1$.

$(1) \Rightarrow (2)$ Suppose that $T$ can be represented as an $n\times n$ upper-triangular matrix with diagonal entries $a_{11}, \dots, a_{nn} \in k$, and let $p(x) = (x-a_{11})\cdots(x-a_{nn})$. Then $p(T)=0$, by the Cayley-Hamilton theorem. 

$(2) \Rightarrow (4)$ See, e.g.,~\cite[Section 12.3, Theorem 22]{DF}.

$(4) \Rightarrow (1)$ A matrix in Jordan canonical form is necessarily upper-triangular.

$(1) \Leftrightarrow (3')$ See, e.g.,~\cite[page 1]{RR}.

$(3) \Leftrightarrow (3')$ A finite set is totally ordered if and only if it is well-ordered.
\end{proof}

Given a field $k$ and a $k$-vector space $V$, we denote by $\End_k(V)$ the $k$-algebra of all linear transformations of $V$.  We define a transformation $T\in \End_k(V)$ to be \emph{triangularizable} if $V$ has a well-ordered basis $(\B, \leq)$ such that $T$ sends each vector $v \in \B$ to the subspace spanned by $\{u \in \B \mid u \leq v\}$. This definition clearly generalizes condition (1) in the above theorem. We then show, in Theorem~\ref{tri-characterization}, that for $T\in \End_k(V)$ (with $k$ and $V$ arbitrary) being triangularizable is equivalent to satisfying condition (3) above, as well as to satisfying each of the following (along with another condition).
\begin{enumerate}
\item[$(2')$] For every finite-dimensional subspace $W$ of $V$ there is a polynomial $p(x) \in k[x]\setminus k$ that factors into linear terms in $k[x]$, such that $p(T)$ annihilates $W$.
\item[$(4')$] $V = \bigoplus_{a \in k} \bigcup_{i=1}^\infty \ker((T-aI)^i)$, where $I \in \End_k(V)$ is the identity transformation.
\end{enumerate}
If $k$ is algebraically closed, then $T$ being triangularizable is also equivalent to the following.
\begin{enumerate}
\item[$(5)$] Every finite-dimensional subspace of $V$ is contained in a finite-dimensional $T$-invariant subspace of $V$.
\item[$(6)$] $V$ is locally artinian, when viewed as a $k[x]$-module, where $x$ acts on $\, V$ as $T$.
\end{enumerate}
Condition $(2')$, of course, is a direct generalization of $(2)$, while $(4')$ is a crude version of $(4)$.

There is extensive literature on triangularization of bounded linear operators on Banach spaces, where condition $(3')$ is taken to be the definition of ``triangularizable" (see~\cite{RR}). This leads to many beautiful results about when collections of operators can be simultaneously triangularized. However, (a statement equivalent to) the stronger condition $(3)$ was chosen as the definition of ``triangularizable" here, since much more of the intuition regarding triangularization, as summarized in Theorem~\ref{cl-tri-thrm}, can be preserved this way. As we show in Example~\ref{weak-tri-eg}, for a transformation $T$ of a vector space, satisfying $(3')$ is generally not equivalent to satisfying $(2')$, $(3)$, and $(4')$.

With the basics of infinite-dimensional triangularization established, we then show that any finite collection of commuting triangularizable elements of $\End_k(V)$ is simultaneously triangularizable (Theorem~\ref{sim-tri-thrm}), generalizing a well-known fact from finite-dimensional linear algebra. Next we show that the inverse, if it exists, of a triangularizable transformation is also triangularizable with respect to the same well-ordered basis, as one would hope (Proposition~\ref{inverse-prop}). Then, after reviewing the standard topology on $\End_k(V)$ in Section~\ref{top-sect}, we characterize, in Theorem~\ref{closure-thrm}, the closure of the set of triangularizable transformations in $\End_k(V)$. In particular, if the field $k$ is algebraically closed, then the closure of this set is $\End_k(V)$, which generalizes the fact, known as Shur's theorem, that over an algebraically closed field every matrix is triangularizable. Then, in Proposition~\ref{top-nilpt-prop}, we give a number of equivalent characterizations of topologically nilpotent transformations  in $\End_k(V)$, i.e., transformations $T$ such that the sequence $(T^i)_{i=1}^\infty$ converges to $0$ in the topology on $\End_k(V)$. These generalize the familiar fact that a matrix is nilpotent if and only if it is similar to a strictly upper-triangular matrix if and only if $0$ is its only eigenvalue (over the algebraic closure of the base field). In Section~\ref{poly-section} we discuss the transformations in $\End_k(V)$ which satisfy a single polynomial on the entire space $V$. Finally, in Theorem~\ref{double-cent-thrm} we generalize to transformations that satisfy a polynomial on the entire space the classical result that the double-centralizer of a matrix is the algebra generated by that matrix.

\section{Preliminaries}

We begin with the following standard fact from finite-dimensional linear algebra, which will be useful for our purposes. The usual proof also works for a vector space of arbitrary dimension, but we provide it here for completeness.

\begin{lemma} \label{lemma:eigenspaces}
Let $k$ be a field, $V$ a $k$-vector space, and $\, T \in \End_k(V)$. Also suppose that $f_1(x), \dots, f_n(x) \in k[x]$ are pairwise relatively prime polynomials, and set $S = f_1(T)\cdots f_n(T)$. Then $\, \kernel (S) = \bigoplus_{i=1}^n \kernel (f_i(T)).$
\end{lemma}

\begin{proof}
For each $i \in \{1, \dots, n\}$ let $$g_i(x) = f_1(x) \cdots f_{i-1}(x)f_{i+1}(x) \cdots f_n(x).$$ Since the $f_i(x)$ are pairwise relatively prime, $\{g_1(x), \dots, g_n(x)\}$ is relatively prime in $k[x]$. Thus there exist $h_1(x), \dots, h_n(x) \in k[x]$ such that $1 = \sum_{i=1}^n h_i(x)g_i(x)$.

Let us now show that the sum $\sum_{i=1}^n \kernel (f_i(T))$ is direct. Thus suppose that $0=\sum_{i=1}^n v_i$ for some $v_i \in \kernel (f_i(T))$. Then for each $j\in \{1, \dots, n\}$ we have $$0 = g_j(T)\bigg(\sum_{i=1}^n v_i \bigg) = g_j(T)(v_j),$$ since $f_1(T), \dots, f_n(T)$ commute with each other. Thus $g_j(T)(v_i) = 0$ for all $i, j \in \{1, \dots, n\}$, and therefore $$v_j = 1\cdot v_j = \sum_{i=1}^n h_i(T)g_i(T)(v_j)= 0$$ for all $j\in \{1, \dots, n\}$. It follows that $$\sum_{i=1}^n \kernel (f_i(T)) = \bigoplus_{i=1}^n \kernel (f_i(T)).$$

Since the $f_i(T)$ commute with each other, clearly $\kernel (S) \supseteq \bigoplus_{i=1}^n \kernel (f_i(T))$, and hence it suffices to show that the reverse inclusion holds. Let $v \in \kernel (S)$. Then $$0 = h_i(T)S(v) = h_i(T)g_i(T)f_i(T)(v) = f_i(T)h_i(T)g_i(T)(v),$$ and hence $h_i(T)g_i(T)(v) \in \ker(f_i(T))$ for all $i \in \{1, \dots, n\}$. Thus $$v = \sum_{i=1}^n h_i(T)g_i(T)(v) \in \bigoplus_{i=1}^n \kernel (f_i(T)),$$ giving the desired conclusion.
\end{proof}

Recall that a binary relation $\leq$ on a set $X$ is a \emph{partial order} if it is reflexive, antisymmetric, and transitive. If, in addition, $x\leq y$ or $y \leq x$ for all $x,y \in X$, then $\leq$ is a \emph{total order}. If, moreover, every non-empty subset of $X$ has a least element with respect to $\leq$, then $\leq$ is a \emph{well order}.

We shall require the following standard set-theoretic fact.

\begin{lemma} \label{well-ord-lemma}
Let $\, (\Lambda, \leq_\Lambda)$ be a well-ordered set, and for each $\lambda \in \Lambda$ let $\, (\Omega_\lambda, \leq_\lambda)$ be a well-ordered set, such that the $\, \Omega_\lambda$ are pairwise disjoint. Define a binary relation $\, \leq$ on $\, \bigcup_{\lambda \in \Lambda} \Omega_\lambda$ as follows: for all $\alpha_1, \alpha_2 \in \bigcup_{\lambda \in \Lambda} \Omega_\lambda$, with $\alpha_1 \in \Omega_{\lambda_1}$ and $\alpha_2 \in \Omega_{\lambda_2}$, let $\alpha_1 \leq \alpha_1$ if either $\lambda_1 = \lambda_2$ and $\alpha_1 \leq_{\lambda_1} \alpha_2$, or $\lambda_1 <_\Lambda \lambda_2$. Then $\, (\bigcup_{\lambda \in \Lambda} \Omega_\lambda, \leq)$ is a well-ordered set.
\end{lemma}

\begin{proof}
It is routine to check that $\leq$ is reflexive, antisymmetric, and transitive. That $\leq$ is a total order then follows from its definition and the fact that $\leq_\Lambda$ and each $\leq_\lambda$ is a total order. Now, let $\Gamma$ be a nonempty subset of $\bigcup_{\lambda \in \Lambda} \Omega_\lambda$. Since $\Lambda$ is well-ordered, there is a least $\lambda \in \Lambda$ (with respect to $\leq_\Lambda$) such that $\Gamma \cap \Omega_\lambda \neq \emptyset$. Since $\Omega_\lambda$ is well-ordered, there is a least $\alpha \in \Omega_\lambda$ (with respect to $\leq_\lambda$) such that $\alpha \in \Gamma \cap \Omega_\lambda$. Then $\alpha$ must be the least element of $\Gamma$ with respect to $\leq$, which shows that $\leq$ is a well order.
\end{proof}

Throughout the paper $\Z$ will denote the set of the integers, $\Z^+$ the set of the positive integers, and $\N$ the set of the natural numbers (including $0$). We shall implicitly rely, whenever appropriate, on the fact that $\Z$ is totally ordered by its usual ordering, while $\Z^+$ and $\N$ are well-ordered.

\section{Triangularization}

We now extend the notion of ``upper-triangular" to transformations of an arbitrary vector space.

Given a subset $X$ of a vector space, we denote by $\langle X \rangle$ the subspace generated by $X$.

\begin{definition} \label{tri-def}
Let $k$ be a field, $V$ a $k$-vector space, $T \in \End_k(V)$, $\B$ a basis for $\, V$, and $\, \leq$ a partial ordering on $\B$. We say that $T$ is \emph{triangular with respect to} $(\B, \leq)$ if $T(v) \in \langle \{u \in \B \mid u \leq v\} \rangle$ for all $v \in \B$, and that $T$ is \emph{strictly triangular with respect to} $(\B, \leq)$ if $T(v) \in \langle \{u \in \B \mid u < v\} \rangle$ for all $v \in \B$.

If $T$ is triangular, respectively strictly triangular, with respect to some \emph{well-ordered} basis for $\, V$, then we say that $T$ is \emph{triangularizable}, respectively \emph{strictly triangularizable}.
\end{definition}

The condition that $T(v) \in \langle \{u \in \B \mid u \leq v\} \rangle$ for all $v \in \B$, in the above definition, is based on the defining property of upper-triangular matrices. We could have used the lower-triangular analog instead: $T(v) \in \langle \{u \in \B \mid v \leq u\} \rangle$ for all $v \in \B$. This would have resulted in an equivalent notion of ``triangular", for given a partially ordered basis $(\B, \leq)$, one has $T(v) \in \langle \{u \in \B \mid u \leq v\} \rangle$ for all $v \in \B$ if and only if $T(v) \in \langle \{u \in \B \mid v \leq' u\} \rangle$ for all $v \in \B$, where $\leq'$ is the opposite partial ordering of $\leq$. (I.e., $v \leq' u$ if and only if $u \leq v$, for all $u, v\in \B$.)

Even though we allowed $\leq$ to be an arbitrary partial order in the above definition, for occasional convenience, our primary interest will be in transformations that are triangular with respect to a well-ordered basis, which is why we reserve ``triangularizable" for that case alone. We focus on this case since, as we shall see, such transformations behave very much like triangular matrices, and to a significantly greater extent than transformations that are triangular with respect to a merely totally ordered basis. Still, it may be of interest to investigate other sorts of partially ordered bases in this context. For example, one could describe the transformations that are triangular with respect to orderings that are opposite to well orderings, which would produce a theory substantially different (and more messy) than the one presented here. (An instance of such a transformation can be found in Example~\ref{lower-eg}.) In order to avoid a lengthy digression, however, we shall not discuss such possibilities in detail.

Sometimes, we shall find it more convenient to index bases with ordered sets rather than ordering the bases themselves, when dealing with triangularization. Also, given $k$-vector spaces $W \subseteq V$ and a transformation $T \in \End_k(V)$ we say that $W$ is $T$-\emph{invariant} if $T(W) \subseteq W$.

In Theorem~\ref{tri-characterization} we shall give a number of equivalent characterizations of triangularizable transformations, but we require a few preliminary results.

\begin{lemma} \label{invar-space-lemma}
Let $k$ be a field, $V$ a $k$-vector space, and $T \in \End_k(V)$. Then $T$ is triangularizable if and only if there exists a well-ordered $\, ($by inclusion$)$ set of $T$-invariant subspaces of $\, V$, which is maximal as a well-ordered set of subspaces of $\, V$.
\end{lemma}

\begin{proof}
Suppose that $T$ is triangularizable. Then there is a well-ordered set $(\Omega, \leq)$ and a basis $\B = \{v_\alpha \mid \alpha \in \Omega\}$ for $V$ such that $T(v_\alpha) \in \langle \{v_\beta \mid \beta \leq \alpha\} \rangle$ for all $\alpha \in \Omega$. Since every well-ordered set is order-isomorphic to an ordinal, we may assume that $\Omega$ is an ordinal. For each $\alpha \in \Omega$ set $V_\alpha = \langle \{v_\beta \mid \beta < \alpha\} \rangle$, where $V_0$ is understood to be the zero space ($0$ being the least element of $\Omega$). Then for all $\alpha_1, \alpha_2 \in \Omega$ we have $V_{\alpha_1} \subseteq V_{\alpha_2}$ if and only if $\alpha_1 \leq \alpha_2$.  Since $(\Omega, \leq)$ is well-ordered, it follows that $X=\{V_\alpha \mid \alpha \in \Omega^+\}$ is well-ordered by set inclusion, where $\Omega^+ = \Omega \cup \{\Omega\}$ is the successor of $\Omega$ and $V = V_{\Omega}$. Moreover, $T(v_\beta) \in V_\alpha$ for all $\alpha, \beta \in \Omega$ satisfying $\beta < \alpha$, from which it follows that each element of $X$ is $T$-invariant. It remains to show that $X$ is maximal. First, note that for each $\alpha \in \Omega^+$ we have $\bigcup_{\beta<\alpha} V_\beta \subseteq V_\alpha$, with equality if $\alpha$ is a limit ordinal, and $V_\alpha/(\bigcup_{\beta<\alpha} V_\beta)$ one-dimensional otherwise.

Now, let $W \subseteq V$ be a subspace that is comparable under set inclusion to $V_\alpha$ for each $\alpha \in \Omega^+$. Since $\Omega^+$ is well-ordered, there is a least $\alpha \in \Omega^+$ such that $W\subseteq V_\alpha$. Since $W$ is comparable to each element of $X$, from the choice of $V_\alpha$ it follows that $V_\beta \subset W$ for all $\beta < \alpha$. If $\alpha$ is a limit ordinal, then $V_\alpha = \bigcup_{\beta<\alpha} V_\beta \subseteq W$, and hence $W=V_\alpha \in X$. Otherwise, there is a $\beta \in \Omega^+$ such that $\alpha$ is the successor of $\beta$, and $V_\beta \subset W \subseteq V_\alpha$. But in this case $V_\alpha /V_\beta$ is one-dimensional, and therefore $W=V_\alpha \in X$ once again. Thus $X$ is a maximal well-ordered set of subspaces of $V$.

Conversely, suppose that there exists a well-ordered set $(\Omega, \leq)$, which we may assume to be an ordinal, and a set $X = \{V_\alpha \mid \alpha \in \Omega\}$ of $T$-invariant subspaces of $V$, such that $V_{\alpha_1} \subseteq V_{\alpha_2}$ if and only if $\alpha_1 \leq \alpha_2$ (for all $\alpha_1, \alpha_2 \in \Omega$), and $X$ is maximal as a well-ordered set of subspaces of $V$. Let $\alpha \in \Omega$ be any element. If $\alpha$ is a successor ordinal, with predecessor $\beta$, then $V_\alpha/V_\beta$ must be one-dimensional, by the maximality of $X$. If $\alpha$ is a limit ordinal, then for all $\beta < \alpha$ we have $V_\beta \subset \bigcup_{\gamma < \alpha} V_\gamma \subseteq V_\alpha$. Again, by the maximality of $X$, this implies that $\bigcup_{\gamma < \alpha} V_\gamma \in X$, and hence $\bigcup_{\gamma < \alpha} V_\gamma = V_\alpha$. Now for each $\beta \in \Omega$ with successor $\alpha \in \Omega$ let $v_\alpha \in V$ be such that $v_\alpha + V_\beta$ spans $V_\alpha/V_\beta$. Also, let $$\Gamma = \{\alpha \in \Omega \mid \alpha \text{ is a successor ordinal}\}.$$ As a subset of a well-ordered set, $\Gamma$ is itself well-ordered by (the restriction of) $\leq$. We claim that $\{v_\alpha \mid \alpha \in \Gamma\}$ is a basis for $V$ with respect to which $T$ is triangular.

Since for each $\alpha \in \Omega$ we have $V_\alpha = \langle \{v_\beta \mid \beta \leq \alpha, \beta \in \Gamma\}\rangle$, and since $X$ must contain $V$, by virtue of being maximal, if follows that $\{v_\alpha \mid \alpha \in \Gamma\}$ spans $V$. From the fact that the spaces $V_\alpha$ are distinct it also follows that $\{v_\alpha \mid \alpha \in \Gamma\}$ is linearly independent, and hence is a basis for $V$. Now let $\alpha \in \Gamma$ be any element. Then $T(v_\alpha) \in V_\alpha$, since $V_\alpha$ is $T$-invariant, and hence $T(v_\alpha) \in \langle \{v_\beta \mid \beta \leq \alpha, \beta \in \Gamma\}\rangle$. That is, $T$ is triangular with respect to $\{v_\alpha \mid \alpha \in \Gamma\}$.
\end{proof}

\begin{proposition} \label{upper-tri-prop}
Let $k$ be a field, $V$ a $k$-vector space, $(\B, \leq)$ a well-ordered basis for $\, V$, and $T \in \End_k(V)$ a transformation triangular with respect to $\B$. Then the following hold.
\begin{enumerate}
\item[$(1)$] If $\, W \subseteq V$ is a finite-dimensional subspace, then $\, W$ is contained in a finite-dimensional $T$-invariant subspace of $\, V$.
\item[$(2)$] There is a partial ordering $\, \preceq$ on $\B$ such that $T$ is triangular with respect to $\, (\B, \preceq)$ and $\, \{u \in \B \mid u \preceq v\}$ is finite for all $v\in \B$.
\end{enumerate} 
\end{proposition}

\begin{proof}
(1) Let $U_1 \subseteq \B$ be a finite subset such that $W \subseteq \langle U_1 \rangle$. Now for each $i > 1$ ($i \in \Z^+$) define recursively $$U_i = \{v \in \B \mid \pi_v T(U_{i-1}) \neq 0\} \cup U_{i-1},$$ where $\pi_v \in \End_k(V)$ is the projection onto $\langle v \rangle$ with kernel $\langle \B \setminus \{v\} \rangle$. Since $U_1$ is finite, it follows by induction that every $U_i$ is finite. Also, we have $T(U_i) \subseteq \langle U_{i+1} \rangle$ for all $i \in \Z^+$.

We claim that the chain $$U_1 \subseteq U_2 \subseteq U_3 \subseteq \cdots$$ must stabilize after finitely many steps. If not, then for each $i \in \Z^+$ let $v_i \in \B$ be the maximal element, with respect to $\leq$, such that $v_i \in U_i \setminus U_{i-1}$ (where $U_0$ is understood to be the empty set). This is well-defined since each $U_i \setminus U_{i-1}$ is finite but nonempty. Then for each $i >1$, there exists $u \in U_{i-1} \setminus U_{i-2}$ such that $\pi_{v_i} T(u) \neq 0$, by the definition of $U_i$. Since $T$ is triangular with respect to $\B$, we have $T(u) \in \langle\{w \in \B \mid w\leq u\}\rangle$, from which it follows that $v_i \leq u \leq v_{i-1}$. Moreover, since $v_i \in U_i \setminus U_{i-1}$ and $v_{i-1} \in U_{i-1}$, we have $v_i < v_{i-1}$. Thus $$v_1 > v_2 > v_3 > \cdots$$ is an infinite strictly descending chain of elements of $\B$. This contradicts $\B$ being well-ordered, since $\{v_1, v_2, v_3, \dots \}$ has no least element. 

Hence there exists $n \in \Z^+$ such that $U_n = U_{n+1}$, and therefore $T(U_n) \subseteq \langle U_{n+1} \rangle = \langle U_n \rangle$. It follows that $T(\langle U_n \rangle) \subseteq \langle U_n \rangle$, where $W \subseteq \langle U_n \rangle$ and $\langle U_n \rangle$ is finite-dimensional, as desired.

(2) Given $u,v \in \B$, we write $u \preceq v$ if either $u=v$ or there exist $w_1, \dots, w_n \in \B$, where $w_1 = v$ and $w_n = u$, such that $$\pi_{w_{n}}T\pi_{w_{n-1}}T\pi_{w_{n-2}} \cdots \pi_{w_2}T\pi_{w_1} \neq 0.$$ (Note that this product being nonzero is equivalent to each of $\pi_{w_{n}}T\pi_{w_{n-1}}, \dots, \pi_{w_2}T\pi_{w_1}$ being nonzero, since the image of each of the projections involved is $1$-dimensional.) Let $v \in \B$, and let $U_1 = \{v\}$. Then defining $U_i$ for all $i >1$ as in the proof of (1), for any $u \in \B$ we have $u \preceq v$ if and only if $u \in U_i$ for some $i \in \Z^+$. Since the chain $$U_1 \subseteq U_2 \subseteq U_3 \subseteq \cdots$$ must stabilize after finitely many steps, and each $U_i$ is finite, it follows that $\{u \in \B \mid u \preceq v\}$ is finite for all $v\in \B$. It remains to show that $\preceq$ is a partial order.

The binary relation $\preceq$ is reflexive, by definition. To show that $\preceq$ is antisymmetric, first we note that given $u,v \in \B$, if $\pi_u T \pi_v \neq 0$, then $u \leq v$, since $T$ is triangular with respect to $(\B, \leq)$, and hence $u \preceq v$ implies that $u \leq v$. (I.e., $\leq$ extends $\preceq$.) Thus, if $u \preceq v$ and $v \preceq u$ for some $u, v \in \B$, then $u \leq v$ and $v \leq u$, from which it follows that $u=v$. Finally, to show that $\preceq$ is transitive, suppose that $u \preceq v$ and $v \preceq w$ for some $u, v, w \in \B$. Then there exist $x_1, \dots, x_n, y_1, \dots, y_m \in \B$, where $x_1 = v = y_m$, $x_n = u$, $y_1 = w$, such that $$\pi_{x_n}T\pi_{x_{n-1}} \cdots \pi_{x_2}T\pi_{x_1} \neq 0 \text{ and } \pi_{y_{m}}T\pi_{y_{m-1}} \cdots \pi_{y_2}T\pi_{y_1} \neq 0.$$ It follows that $$\pi_{u}T\pi_{x_{n-1}} \cdots \pi_{x_2}T(v) = au \text{ and } \pi_{v}T\pi_{y_{m-1}} \cdots \pi_{y_2}T(w) = bv$$ for some $a,b \in k \setminus \{0\}$. Therefore, $$\pi_{u}T\pi_{x_{n-1}} \cdots \pi_{x_2}T\pi_{v} T\pi_{y_{m-1}} \cdots \pi_{y_2}T(w) = abu \neq 0,$$ and hence $u \preceq w$, as required.
\end{proof}

\begin{lemma} \label{upper-tri-lemma}
Let $k$ be a field, $V$ a $k$-vector space, and $T \in \End_k(V)$. If $\, V = \bigcup_{i=1}^\infty \kernel (T^i)$, then $T$ is strictly triangularizable.
\end{lemma}

\begin{proof}
For each $i \geq 1$ let $U_i \subseteq \ker (T^i)\setminus \ker (T^{i-1})$ be a linearly independent set such that $U_i + \ker (T^{i-1})$ is a basis for $\ker (T^i)/ \ker (T^{i-1}),$ where $U_i$ is possibly empty. Then $\bigcup_{i=1}^n U_i$ is a basis for $\ker(T^n)$, for each $n \in \Z^+$, and hence $\B = \bigcup_{i=1}^\infty U_i$ is a basis for $V$.

Now for each $i\in \Z^+$, let $\leq_i$ be a well-ordering on $U_i$ (chosen arbitrarily), and define a binary relation $\leq_{\B}$ on $\B$ as follows. Given $u_1, u_2 \in \B$, where $u_1 \in U_{i_1}$ and $u_2 \in U_{i_2}$ ($i_1,i_2 \in \Z^+$), let $u_1 \leq_{\B} u_2$ if either $i_1=i_2$ and $u_1 \leq_{i_1} u_2$, or $i_1 < i_2$. Then, by Lemma~\ref{well-ord-lemma}, $(\B, \leq_{\B})$ is well-ordered.

Finally, let $u \in \B$, and let $n \in \Z^+$ be such that $u \in U_n \, (\subseteq \ker(T^n))$. Then $$T (u) \in \kernel (T^{n-1}) = \bigg\langle \bigcup_{i=1}^{n-1} U_i \bigg\rangle \subseteq \langle \{v \in \B \mid v <_{\B} u\} \rangle,$$ by the definition of $\leq_{\B}$, which shows that $T$ is strictly triangular with respect to $\B$.
\end{proof}

We are now ready for our main result, which characterizes the triangularizable transformations of an arbitrary vector space. The conditions (1)--(3) in the statement generalize the corresponding ones in Theorem~\ref{cl-tri-thrm}, while condition (4) is a crude version of Theorem~\ref{cl-tri-thrm}(4).

We recall that given a commutative ring $R$, an $R$-module $M$ is called \emph{locally artinian} if every finitely-generated $R$-submodule of $M$ is artinian.

\begin{theorem} \label{tri-characterization}
Let $k$ be a field, $V$ a $k$-vector space, and $T \in \End_k(V)$. Then the following are equivalent.
\begin{enumerate}
\item[$(1)$] $T$ is triangularizable.
\item[$(2)$] For every finite-dimensional subspace $\, W$ of $\, V$ there is a polynomial $p(x) \in k[x]\setminus k$ that factors into linear terms in $k[x]$, such that $p(T)$ annihilates $\, W$.
\item[$(3)$] There exists a well-ordered set of $T$-invariant subspaces of $\, V$, which is maximal as a well-ordered set of subspaces of $\, V$.
\item[$(4)$] $V = \bigoplus_{a \in k} \bigcup_{i=1}^\infty \ker((T-aI)^i)$, where $I \in \End_k(V)$ is the identity transformation.
\item[$(5)$] There is a partially ordered basis $\, (\B, \preceq)$ for $\, V$ such that $T$ is triangular with respect to $\, (\B, \preceq)$ and $\, \{u \in \B \mid u \preceq v\}$ is finite for all $v\in \B$.
\end{enumerate}
Moreover, if $k$ is algebraically closed, then these are also equivalent to the following.
\begin{enumerate}
\item[$(6)$] Every finite-dimensional subspace of $\, V$ is contained in a finite-dimensional $T$-invariant subspace of $\, V$.
\item[$(7)$] $V$ is locally artinian, when viewed as a $k[x]$-module, where $x$ acts on $\, V$ as $T$.
\end{enumerate}
\end{theorem}

\begin{proof}
By Proposition~\ref{upper-tri-prop}, (1) implies (5) and (6). Also, (1) and (3) are equivalent, by Lemma~\ref{invar-space-lemma}. We shall prove that $(5) \Rightarrow (2) \Rightarrow (4) \Rightarrow (1)$, and then treat (6) and (7) at the end.

$(5) \Rightarrow (2)$ Let $(\B, \preceq)$ be as in (5), and let $W$ be a finite-dimensional subspace of $V$. We can find a finite subset $X$ of $\B$ such that $W \subseteq \langle X \rangle$. Then, by hypothesis, the set $$Y = \{u \in \B \mid \exists v \in X \, (u \preceq v)\}$$ is finite. Since $\preceq$ is transitive, for all $u \in Y$ and $v \in \B$ such that $v \preceq u$, we have $v \in Y$. Hence, the assumption that $T$ is triangular with respect to $(\B, \preceq)$ implies that $\langle Y \rangle$ is $T$-invariant. By the order-extension principle, the restriction of $\preceq$ to $Y \subseteq \B$ can be extended to a total order $\leq$ on $Y$. Then $$T(v) \in \langle \{u \in Y \mid u \preceq v\} \rangle \subseteq \langle \{u \in Y \mid u \leq v\} \rangle$$ for all $v \in Y$. Hence the restriction of $T$ to $\langle Y \rangle$ is triangular with respect to $(Y, \leq)$, and so can be represented as a (finite) upper-triangular matrix. Therefore, by Theorem~\ref{cl-tri-thrm}, there is a polynomial $p(x) \in k[x]\setminus k$ that factors into linear terms in $k[x]$, such that $p(T)$ annihilates $\langle Y \rangle$ and hence also $W \subseteq \langle X \rangle \subseteq \langle Y \rangle$, proving (2).

$(2) \Rightarrow (4)$ Suppose that (2) holds, and let $P \subseteq k[x]\setminus k$ be the subset consisting of all the polynomials that factor into linear terms. Then $$V = \bigcup_{p \in P} \kernel (p(T)) = \bigoplus_{a \in k} \bigcup_{i=1}^\infty \ker((T-aI)^i),$$ by Lemma~\ref{lemma:eigenspaces}, and hence (4) holds. 

$(4) \Rightarrow (1)$ Suppose that (4) holds. Upon well-ordering $k$, by Lemma~\ref{well-ord-lemma}, to prove (1), it is enough to show that for each $a \in k$ there is a well-ordered basis for $\bigcup_{i=1}^\infty \ker((T-aI)^i)$ with respect to which $T$ is triangular (when restricted to the $T$-invariant subspace $\bigcup_{i=1}^\infty \ker((T-aI)^i)$ of $V$). Thus, let us assume that $V = \bigcup_{i=1}^\infty \ker((T-aI)^i)$ for some $a \in k$, and let $S = T-aI$. Then, by Lemma~\ref{upper-tri-lemma}, there is a well-ordered basis $(\B, \leq)$ for $V$, with respect to which $S$ is triangular. Thus, for any $v \in \B$ we have $$T(v) = S(v) + av \in \langle \{u \in \B \mid u \leq v\} \rangle,$$ showing that $T$ is triangular with respect to $(\B, \leq)$. 

We have shown that (1)--(5) are equivalent. Next, let us suppose that $k$ is algebraically closed and that (6) holds, and show that (2) also holds. Let $W$ be a finite-dimensional subspace of $V$. Then, by (6), there is a finite-dimensional $T$-invariant subspace $W'$ of $V$ containing $W$. Viewing the restriction of $T$ to $W'$ as a (finite) matrix, there is a polynomial $p(x) \in k[x]\setminus k$ such that $p(T)$ annihilates $W'$ (e.g., by the Cayley-Hamilton theorem). In particular, $p(T)$ annihilates $W$. Since $k$ is algebraically closed, $p(x)$ factors into linear terms in $k[x]$, showing that (2) holds. Thus, when $k$ is algebraically closed, (1)--(6) are equivalent.

To conclude the proof, we shall show that (6) and (7) are equivalent. Thus suppose that (6) holds, and let $M$ be a finitely-generated $k[x]$-submodule of $V$, where $x$ acts as $T$. Let $W \subseteq M$ be a finite set such that $M=k[x]W$. Then $W$ is contained in a finite-dimensional $T$-invariant subspace $M'$ of $V$, by (6). But $T$-invariant subspaces of $V$ are precisely the $k[x]$-submodules of $V$, which shows that $M$ is contained in the finite-dimensional $k[x]$-submodule $M'$ of $V$. Thus $M$ is finite-dimensional as a $k$-vector space, and hence artinian, proving (7).

Conversely, suppose that (7) holds, and let $W$ be a finite-dimensional subspace of $V$. Then, by (7), the $k[x]$-submodule $M=k[x]W$ of $V$ is artinian. Since $M$ is a $T$-invariant subspace of $V$, to conclude that (6) holds it suffices to show that $M$ is finite-dimensional. But since $k[x]$ is a principal ideal domain and $M$ is a finitely-generated $k[x]$-module, $$M \cong k[x]^r \oplus k[x]/\langle f_1(x)\rangle \oplus \cdots \oplus k[x]/\langle f_n(x)\rangle,$$ where $r \in \N$, $f_1(x), \dots, f_n(x) \in k[x]\setminus \{0\}$, and $\langle f_i(x)\rangle$ denotes the ideal of $k[x]$ generated by $f_i(x)$. (See, e.g.,~\cite[Section 12.1, Theorem 5]{DF}.) Since $M$ is artinian, we must have $r=0$, and hence $M$ is finite-dimensional as a $k$-vector space, giving the desired conclusion.
\end{proof}

As mentioned in the Introduction, in the literature on bounded linear operators on Banach spaces, a transformation $T$ is said to be ``triangularizable" if there is a chain (i.e., totally ordered set) of $T$-invariant subspaces of the Banach space which is maximal as a chain of subspaces (see~\cite[Definition 7.1.1]{RR}). That is, for such operators, condition ($3'$) from Theorem~\ref{cl-tri-thrm} is used to generalize the notion of ``triangular" from finite-dimensional spaces to infinite-dimensional ones. By using the stronger condition (3) instead (which, by the previous theorem, is equivalent to $T$ being triangularizable, as we have defined the term) in our generalization of ``triangular" we acquire much greater control over the behavior of transformations, as the next example demonstrates. 

To facilitate the discussion, we say that a transformation $T$ of a vector space $V$ is \emph{chain-triangularizable} if there is a chain of $T$-invariant subspaces of $V$, which is maximal as a chain of subspaces of $V$.

\begin{example} \label{weak-tri-eg}
Let $k$ be a field and $V$ a $k$-vector space with basis $\{v_i \mid i \in \Z\}$. Define $T \in \End_k(V)$ by $T(v_i) = v_{i-1}$ for each $i \in \Z$, and extend linearly to all of $V$.  Also for each $i \in \Z$ let $V_i = \langle \{v_j \mid j \leq i\} \rangle$. Then $$\cdots \subseteq V_{-1} \subseteq V_{0} \subseteq V_{1} \subseteq \cdots$$ is a maximal chain of subspaces of $V$ (since every $V_i/V_{i-1}$ is $1$-dimensional, $V = \bigcup_{i\in \Z} V_i$, and $0 = \bigcap_{i\in \Z} V_i$), each $T$-invariant. Thus $T$ is chain-triangularizable. However, $T$ satisfies none of the seven conditions in Theorem~\ref{tri-characterization}. To see this, let $W = \langle v_0 \rangle$. Then any $T$-invariant subspace of $V$ that contains $W$ must contain $v_{-1}, v_{-2}, \dots$, and hence also $V_0$. Therefore $T$ does not satisfy condition (6) in Theorem~\ref{tri-characterization}, and is hence not triangularizable, by Proposition~\ref{upper-tri-prop}. It follows that $T$ does not satisfy any of the conditions (1)--(7) in Theorem~\ref{tri-characterization}. \hfill $\Box$
\end{example}

Let us next derive a useful consequence of Theorem~\ref{tri-characterization}.

Given $k$-vector spaces $W\subseteq V$ and a transformation $T \in \End_k(V)$, we denote by $T|_W$ the restriction of $T$ to $W$.

\begin{corollary}\label{restrict-cor}
Let $k$ be a field, $V$ a $k$-vector space, $T \in \End_k(V)$ triangular with respect to some well-ordered basis for $\, V$, and $\, W\subseteq V$ a $T$-invariant subspace. 
\begin{enumerate}
\item[$(1)$] $T|_W$ is triangular with respect to some well-ordered basis for $\, W$.
\item[$(2)$] Let $\overline{T} \in \End_k(V/W)$ be the transformation defined by $\overline{T}(v+W)=T(v)+W$. Then $\overline{T}$ is triangular with respect to some well-ordered basis for $\, V/W$.
\end{enumerate}
\end{corollary}

\begin{proof}
(1) Let $U \subseteq W$ be a finite-dimensional subspace. Since $U \subseteq V$ and $T$ is triangularizable, by Theorem~\ref{tri-characterization}, there is a polynomial $p(x) \in k[x]\setminus k$ that factors into linear terms in $k[x]$, such that $p(T)$ annihilates $U$. Since $T(W) \subseteq W$, we have $p(T)|_W = p(T|_W)$, and hence $p(T|_W)$ annihilates $U$. Therefore, by Theorem~\ref{tri-characterization}, $T|_W$ is triangular with respect to some well-ordered basis for $W$.

(2) First, note that $\overline{T}$ is well-defined. For if $v_1+W=v_2+W$ for some $v_1, v_2 \in V$, then $v_1-v_2 =w$ for some $w\in W$. Hence $$\overline{T}(v_1+W) = T(v_1)+W = T(v_2)+T(w)+W = T(v_2)+W = \overline{T}(v_2+W).$$ It is routine to verify that $\overline{T}$ is also linear.

Now, let $W \subseteq U \subseteq V$ be a subspace such that $U/W$ is finite-dimensional. Then there exist $u_1, \dots, u_n \in U$ such that $U/W = \langle u_1+W, \dots, u_n+W\rangle$. Since $T$ is triangularizable, by Theorem~\ref{tri-characterization}, there is a polynomial $p(x) \in k[x]\setminus k$ that factors into linear terms in $k[x]$, such that $p(T)$ annihilates $\langle u_1, \dots, u_n\rangle \subseteq V$. Then for any $i \in \{1, \dots, n\}$ we have $$p(\overline{T})(u_i+W) = p(T)(u_i)+W = W,$$ showing that $p(\overline{T})$ annihilates $U/W$. Thus, $\overline{T}$ is triangular with respect to some well-ordered basis for $V/W$, by Theorem~\ref{tri-characterization}.
\end{proof}

To complement our description of triangularizable transformations we also give a more specialized description of diagonalizable ones. This is part of~\cite[Proposition 4.13]{IMR}, but we present a more direct proof here. 

\begin{proposition}\label{prop:operator}
Let $k$ be a field, $V$ a $k$-vector space, and $T \in \End_k(V)$. Then the following are equivalent.
\begin{itemize}
\item[$(1)$] $T$ is diagonalizable. $($I.e., there is a basis for $V$ consisting of eigenvectors of $T$.$)$
\item[$(2)$] For every finite-dimensional subspace $\, W$ of $\, V$ there is a polynomial $p(x) \in k[x]\setminus k$ that factors into \emph{distinct} linear terms in $k[x]$, such that $p(T)$ annihilates $\, W$.
\end{itemize}
\end{proposition}

\begin{proof}
Suppose that $T$ is diagonalizable, and let $\B$ be a basis for $V$ consisting of eigenvectors of $T$. To prove (2), let $W \subseteq V$ be a finite-dimensional subspace. Then we can find $v_1, \dots, v_n \in \B$ such that $W \subseteq \langle v_1, \dots, v_n \rangle$. By hypothesis, for each $i \in \{1, \dots, n\}$ there exists $a_i \in k$ such that $T(v_i) = a_i v_i$. Let $a_1, \dots, a_l$ be the distinct elements of $\{a_1, \dots, a_n\}$ (upon reindexing, if necessary), and set $S = (T-a_1I)\cdots (T-a_lI)$, where $I \in \End_k(V)$ the identity transformation. Then, since the factors $T-a_iI$ commute with each other, $S(\langle v_1, \dots, v_n \rangle) = 0$. Thus letting $p(x) = (x-a_1)\cdots (x-a_l)$, we have $p(T)(W) = 0$.

Conversely, suppose that (2) holds, for each $a \in k$ let $\B_a$ be a basis for $\kernel (T-a I)$, and let $\B = \bigcup_{a \in k} \B_a$. Then, by Lemma~\ref{lemma:eigenspaces}, $\B$ is a basis for $V$, and clearly $\B$ consists of eigenvectors of $T$.
\end{proof}

\section{Simultaneous Triangularization}

Our next goal is to show that any finite commuting collection of triangularizable transformations is triangular with respect to a common well-ordered basis. This generalizes the classical fact that any commuting collection of triangularizable transformations of a finite-dimensional vector space is upper-triangular with respect to some basis for that vector space.

The following notation and observations will be useful.

\begin{definition}
Given a ring $R$ and a subset $X \subseteq R$ we denote by $C_R(X)$ $($or $C(X)$, if there is no danger of ambiguity$)$ the \emph{centralizer} $($or \emph{commutant}$)$  $$\{r \in R \mid rx = xr \mbox{ for all } x \in X\}$$ of $X$ in $R$. Given $r \in R$ we shall also write $C_R(r)$ to mean $C_R(\{r\})$.
\end{definition}

\begin{lemma}\label{nontriv-invar-subspace}
Let $k$ be a field, $V$ a nonzero $k$-vector space, and $T \in \End_k(V)$ triangularizable. Then there exists $a \in k$ such that $\, W = \ker(T-aI)$ is nonzero, where $I \in \End_k(V)$ is the identity transformation. Moreover any such $\, W$ satisfies $C(T)(W) \subseteq W$.
\end{lemma}

\begin{proof}
Let $(\B, \leq)$ be a well-ordered basis with respect to which $T$ is triangular. Since $V \neq 0$, we have $\B \neq \emptyset$. Let $v\in \B$ be the least element with respect to $\leq$. Then $T(v) \in \langle v \rangle$, and hence $T(v) = av$ for some $a \in k$. Letting $W = \ker(T-aI)$, we see that $W \neq 0$, since $v \in W$. Now let $S \in C(T)$ be any element. Then $$(T-aI)S(W) = S(T-aI)(W) = 0,$$ and hence $S(W) \subseteq W$. It follows that $C(T)(W) \subseteq W$.
\end{proof}

\begin{lemma}\label{1-dim-invar-subspace}
Let $k$ be a field, $V$ a nonzero $k$-vector space, and $X \subseteq \End_k(V)$ a finite commutative collection of transformations. If each element of $X$ is triangularizable, then there exists a $1$-dimensional subspace $\, W \subseteq V$ such that $X(W) \subseteq W$.
\end{lemma}

\begin{proof}
Write $X = \{T_1, \dots, T_n\}$. It suffices to construct a nonzero subspace of $V$ on which each $T_i$ acts as a scalar multiple of the identity, since any subspace $U$ of such a space would satisfy $X(U) \subseteq U$, and in particular, any $1$-dimensional subspace.

By Lemma~\ref{nontriv-invar-subspace}, there exists $a_1 \in k$ such that $W_1 = \ker(T_1-a_1I)$ satisfies $0 \neq W_1$ and $X(W_1) \subseteq W_1$. In particular, $T_1$ acts as a scalar multiple of the identity on $W_1$. By Corollary~\ref{restrict-cor}, the restriction $X|_{W_1}$ of $X$ to $W_1$ is a commutative collection of transformations in $\End_k(W_1)$, each triangularizable. Applying Lemma~\ref{nontriv-invar-subspace} again, we find $a_2 \in k$ such that $W_2 = \ker(T_2|_{W_1}-a_2I) \subseteq W_1$ satisfies $0 \neq W_2$ and $X|_{W_1}(W_2) \subseteq W_2$, and hence also $X(W_2) \subseteq W_2$. Now both $T_1$ and $T_2$ act as scalar multiples of the identity on $W_2$. Continuing in this fashion, the construction will yield a nonzero subspace $W_m$ of $V$ ($m \leq n)$ on which every $T_i$ acts this way.
\end{proof}

\begin{theorem} \label{sim-tri-thrm}
Let $k$ be a field, $V$ a $k$-vector space, and $X \subseteq \End_k(V)$ a finite commutative collection of transformations. If each element of $X$ is triangularizable, then there exists a well-ordered basis for $\, V$ with respect to which every element of $X$ is triangular.
\end{theorem}

\begin{proof}
We begin by constructing recursively for each ordinal $\alpha$ a subspace $V_\alpha \subseteq V$ that is invariant under $X$, and for each successor ordinal $\alpha$ a vector $v_\alpha \in V$. Set $V_0=0$. Now let $\alpha$ be an ordinal and assume that $V_\gamma$ has been defined for every $\gamma < \alpha$. If $\alpha$ is a limit ordinal, then let $V_\alpha = \bigcup_{\gamma < \alpha} V_\gamma$. Since each $V_\gamma$ is assumed to be invariant under $X$, their union $V_\alpha$ will also be invariant under $X$. Next, if $\alpha$ is a successor ordinal, then let $\beta$ be its predecessor. By Corollary~\ref{restrict-cor}, the transformation on $V/V_\beta$ induced by each element of $X$ is triangular with respect to some basis for $V/V_\beta$. Thus, by Lemma~\ref{1-dim-invar-subspace}, there is a $1$-dimensional subspace $W/V_\beta$ of $V/V_\beta$ invariant under each transformation on $V/V_\beta$ induced by an element of $X$ (assuming that $V\neq V_\beta$). Let $v_\alpha \in V$ be such that $\{v_\alpha+V_\beta\}$ is a basis for $W/V_\beta$, and define $V_\alpha = \langle V_\beta \cup \{v_\alpha\}\rangle$. Then $V_\alpha$ must be invariant under $X$, because of the invariance of $V_\beta$ and $W$. We proceed in this fashion until $V = \bigcup_{\alpha \in \Lambda} V_\alpha$ for some ordinal $\Lambda$.

Now let $$\Gamma = \{\alpha \in \Lambda \mid \alpha \text{ is a successor ordinal}\},$$ and let $\B = \{v_\alpha \mid \alpha \in \Gamma\}$. Since we introduced new vectors only at successor steps in our construction, $$V = \bigcup_{\alpha \in \Gamma} V_\alpha =  \bigcup_{\alpha \in \Gamma} \langle \{v_\gamma \mid \gamma \leq \alpha, \gamma \in \Gamma\}\rangle,$$ and hence $V = \langle \B \rangle$. Since $V_\alpha/V_\beta = \langle v_\alpha +V_\beta \rangle$ is $1$-dimensional for all $\alpha \in \Gamma$ with predecessor $\beta$, we conclude that $\B$ is a basis for $V$. Also, since $V_\alpha = \langle \{v_\gamma \mid \gamma \leq \alpha, \gamma \in \Gamma\}\rangle$ is invariant under $X$ for all $\alpha \in \Gamma$, it follows that $X(v_\alpha) \in \langle \{v_\gamma \mid \gamma \leq \alpha, \gamma \in \Gamma\} \rangle$ for all $\alpha \in \Gamma$. Thus, every element of $X$ is triangular with respect to $\B$, a basis for $V$ indexed by the well-ordered set $\Gamma$.
\end{proof}

In~\cite[Example 4.17]{IMR} there is a construction of a countably infinite commutative set $E$ of transformations of a countably infinite-dimensional vector space $V$, over an arbitrary field, such that each transformation in $E$ is diagonalizable (an idempotent, actually), but such that no $1$-dimensional subspace of $V$ is invariant under $E$. Thus, there is no well-ordered basis for $V$ with respect to which every element of $E$ is triangular, since the least element of such a basis would be an eigenvector of every element of $E$. Hence, Theorem~\ref{sim-tri-thrm} cannot be extended to arbitrary infinite commutative collections of triangularizable transformations.

\section{Inverses}

In the following proposition we generalize the facts that an upper-triangular matrix is invertible if and only if it has only nonzero diagonal entries, and that the inverse of an upper-triangular matrix is also upper-triangular. These are simple observations, but they further reinforce the idea that our notion of ``triangularizable" preserves intuition from finite-dimensional linear algebra.

\begin{proposition} \label{inverse-prop}
Let $k$ be a field, $V$ a $k$-vector space, and $T \in \End_k(V)$ triangular with respect to some well-ordered basis $\, (\B, \leq)$ for $\, V$. Also for each $v \in \B$ let  $\pi_v \in \End_k(V)$ be the projection onto $\, \langle v \rangle$ with kernel $\, \langle \B \setminus \{v\} \rangle$. Then the following are equivalent.
\begin{enumerate}
\item[$(1)$] $T$ is invertible.
\item[$(2)$] The restriction of $T$ to any finite-dimensional $T$-invariant subspace of $\, V$ is invertible.
\item[$(3)$] $T$ is injective.
\item[$(4)$] $T(\langle u \in \B \mid u \leq v \rangle) = \langle u \in \B \mid u \leq v \rangle$ for all $v \in \B$.
\item[$(5)$] $\pi_vT\pi_v \neq 0$ for all $v \in \B$.  
\end{enumerate}
Moreover, if $T$ is invertible, then its inverse is triangular with respect to $\, (\B, \leq)$.
\end{proposition}

\begin{proof}
We shall show that $(1) \Rightarrow (2) \Rightarrow (3) \Rightarrow (4) \Rightarrow (5) \Rightarrow (3) \Rightarrow (1)$. For the rest of the proof let $U_v = \langle u \in \B \mid u \leq v \rangle$ for each $v \in \B$.

$(1) \Rightarrow (2)$ Let $W$ be a finite-dimensional $T$-invariant subspace of $V$. If $T$ is invertible, then $\ker(T) = 0$, and hence also $\ker(T|_W)=0$. Standard finite-dimensional linear algebra then gives that $T|_W$ is invertible.

$(2) \Rightarrow (3)$ Suppose that $T(v) = 0$ for some $v \in V$. Then by Proposition~\ref{upper-tri-prop}, $v$ is an element of some finite-dimensional $T$-invariant subspace $W$ of $V$. Now, by (2) $T|_W$ is invertible, and therefore $T(v)=0=T|_W(v)$ implies that $v=0$. Thus $T$ is injective.

$(3) \Rightarrow (4)$ Suppose that $T$ is injective. Since $T$ is triangular with respect to $\B$, we have $T(U_v) \subseteq \bigcup_{u\leq v} U_u = U_v$ for all $v \in \B$.

Now suppose that $U_w \not\subseteq T(U_w)$ for some $w \in \B$. Since $\B$ is well-ordered, we may assume that $w$ is the least element of $\B$ with this property. Thus for all $v\in \B$ such that $v < w$, we have $v \in U_v = T(U_v) \subseteq T(U_w)$, and therefore $w \not\in T(U_w)$. Since $T$ is triangular with respect to $\B$ this implies that $T(w) \in \langle u \in \B \mid u < w \rangle$. Thus either $T(w)=0$ or $T(w) \in U_v$ for some $v \in \B$ such that $v < w$. But, by hypothesis, $U_v = T(U_v)$ for any $v < w$, and hence either $T(w)=0$ or $T(u)=T(w)$ for some $u \in U_v$, both of which would contradict $T$ being injective. Therefore $T(U_v) = U_v$ for all $v \in \B$.

$(4) \Rightarrow (5)$ Suppose that $T(U_v) = U_v$ for all $v \in \B$. Then, given any $v \in \B$, there exists $w \in U_v$ such that $T(w)=v$. Write $w = \sum_{i=1}^n a_i u_i$ for some $a_i \in k$ and $u_i\in \B$, such that $u_1 < u_2 < \dots < u_n = v$. Since $T$ is triangular with respect to $\B$, we have $T(\sum_{i=1}^{n-1} a_i u_i) \in U_{u_{n-1}}$ (if $n > 1$). Hence $$v=\pi_{v}T(w) = \pi_{v}T\bigg(\sum_{i=1}^{n-1} a_i u_i \bigg) + a_n\pi_{v}T(u_n) =  a_n\pi_{v}T(v),$$ and therefore $\pi_vT\pi_v \neq 0$.

$(5) \Rightarrow (3)$ Let $w \in V \setminus \{0\}$, and suppose that $T(w)=0$. Write $w = \sum_{i=1}^n a_i v_i$ for some $a_i \in k\setminus \{0\}$ and $v_i\in \B$, such that $v_1 < v_2 < \dots < v_n$. Since $T$ is triangular with respect to $\B$, we have $T(\sum_{i=1}^{n-1} a_i v_i) \in U_{v_{n-1}}$ (if $n > 1$). Hence $$0=\pi_{v_n}T(w) = \pi_{v_n}T\bigg(\sum_{i=1}^{n-1} a_i v_i \bigg) + a_n\pi_{v_n}T(v_n) =  a_n\pi_{v_n}T(v_n),$$ which implies that $\pi_vT\pi_v = 0$, since $a_n \neq 0$. Thus if $\pi_vT\pi_v \neq 0$ for all $v \in \B$, then $T$ must be injective.

$(3) \Rightarrow (1)$ Supposing that $T$ is injective, we also have $T(U_v) = U_v$ for all $v \in \B$, by $(3) \Rightarrow (4)$. Thus, $v \in T(U_v)$ for all $v \in \B$, and therefore $\B \subseteq T(V)$, which implies that $T$ is surjective. Therefore $T$ is a bijection. The desired conclusion now follows from the easy fact that the inverse of any $k$-linear bijection from $V$ to $V$ is necessarily $k$-linear.

For the final claim, suppose that $T$ is invertible, with inverse $T^{-1} \in \End_k(V)$. Then we have $T(U_v) = U_v$ for all $v\in \B$, by the equivalence of $(1)$ and $(3)$. Therefore $T^{-1}(U_v) = U_v$ for all $v\in \B$, and in particular $T^{-1}(v) \in U_v$. Hence $T^{-1}$ is also triangular with respect to $\B$.
\end{proof}

We note that a triangularizable transformation can be surjective without being invertible, in contrast to the situation with injectivity discussed above. For example, let $k$ be a field, $V$ a $k$-vector space with basis $\B = \{v_i \mid i \in \N\}$, and $T \in \End_k(V)$ such that $T(v_0)=0$ and $T(v_i)=v_{i-1}$ for all $i \geq 1$. Then clearly $T$ is (strictly) triangular with respect to $\B$ and surjective, but it is not injective.

The next two examples show that chain-triangularizable transformations are not nearly as well-behaved with respect to inversion as triangularizable ones.

\begin{example}
Let $k$ be a field and $V$ a $k$-vector space with basis $\{v_i \mid i \in \Z\}$. Define $T \in \End_k(V)$ by $T(v_i) = v_{i-1}$ for each $i \in \Z$, and extend linearly to all of $V$. As seen in Example~\ref{weak-tri-eg}, $T$ is chain-triangularizable but not triangularizable. Clearly $T$ is invertible, with inverse $T^{-1}$ defined by $T^{-1}(v_i) = v_{i+1}$ for all $i \in \Z$. 

Letting $V_i = \langle \{v_j \mid j \leq i\} \rangle$ for each $i \in \Z$, as we showed in Example~\ref{weak-tri-eg}, $$\cdots \subseteq V_{-1} \subseteq V_{0} \subseteq V_{1} \subseteq \cdots$$ is a maximal chain of $T$-invariant subspaces of $V$, each $T$-invariant. However none of the $V_i$ is $T^{-1}$-invariant, since $T^{-1}(v_i) = v_{i+1} \notin V_i$ for each $i \in \Z$. Thus, a ``triangularizing chain" for $T$ need not be one for $T^{-1}$. Or, to put it another way, $T$ is triangular with respect to the totally ordered basis $\{v_i \mid i \in \Z\}$, but $T^{-1}$ is not. \hfill $\Box$
\end{example}

\begin{example} \label{lower-eg}
Let $k$ be a field and $V$ a $k$-vector space with basis $\{v_i \mid i \in \N\}$. Define $T \in \End_k(V)$ by $T(v_i) = v_{i+1}$ for each $i \in \N$, and extend linearly to all of $V$. Also for each $i \in \N$ let $V_i = \langle \{v_j \mid j \geq i\} \rangle$. Then $$V_{0} \supseteq V_{1} \supseteq V_{2} \supseteq \cdots$$ is a maximal chain of subspaces of $V$, each $T$-invariant (by the same argument as in Example~\ref{weak-tri-eg}). Thus $T$ is chain-triangularizable. On the other hand, $T$ is not triangularizable, by Proposition~\ref{upper-tri-prop}, since the only finite-dimensional $T$-invariant subspace of $V$ is the zero space.

Now, given the previous observation, it is vacuously true that the restriction of $T$ to any finite-dimensional $T$-invariant subspace of $V$ is invertible. But unlike triangularizable transformations with this property, $T$ itself is certainly not invertible, since $v_0 \notin T(V)$, and hence $T$ is not surjective.  \hfill $\Box$
\end{example}

\section{Topology} \label{top-sect}

We begin this section by recalling the standard topology on the ring $\End_k(V)$, which will help us with subsequent results.

Let $X$ and $Y$ be sets, and let $Y^X$ denote the set of all functions $X \to Y$.
The \emph{function} (or \emph{finite}) \emph{topology} on $Y^X$ has a base of open sets of the following form: $$\{f \in Y^X \mid f(x_1) = y_1, \dots, f(x_n) = y_n\} \ (x_1, \dots, x_n \in X, y_1, \dots, y_n \in Y).$$ It is straightforward to see that this coincides with the product topology on $Y^X = \prod_X Y$, where each component set $Y$ is given the discrete topology.  As a product of discrete spaces, this space is Hausdorff.

Now let $V$ be a vector space over a field $k$. Then $\End_k(V) \subseteq V^V$ inherits a topology from the function topology on $V^V$, which we shall also call the \emph{function topology}. Under this topology $\End_k(V)$ is a topological ring (see, e.g.,~\cite[Theorem~29.1]{Warner}), i.e., a ring $R$ equipped with a topology that makes $+:R\times R \to R$, $-:R \to R$, and $\cdot :R\times R \to R$ continuous. Alternatively, we may describe the function topology on $\End_k(V)$ as the topology having a base of open sets of the following form: $$\{S \in \End_k(V) \mid S|_W=T|_W\} \ (T \in \End_k(V), W \subseteq V \text{ a finite-dimensional subspace}).$$ Observe that when $V$ is finite-dimensional, $\End_k(V)$ is discrete in this topology.

Next, we describe the closure of the set of triangularizable transformations in $\End_k(V)$ with respect to the above topology. This result generalizes (as did Theorem~\ref{tri-characterization}) Shur's theorem, which says that every (finite) matrix over an algebraically closed field is triangularizable.

\begin{theorem} \label{closure-thrm}
Let $k$ be a field and $V$ a $k$-vector space. Define $\T \subseteq \End_k(V)$ to be the subset of all triangularizable transformations, and let $\overline{\T} \subseteq \End_k(V)$ be the closure of $\T$ in the function topology. 

Then for all $T \in \End_k(V)$, we have $T \in \overline{\T}$ if and only if the restriction of $T$ to any finite-dimensional $T$-invariant subspace of $\, V$ is triangularizable. In particular, if $k$ is algebraically closed, then $\overline{\T} = \End_k(V)$.
\end{theorem}

\begin{proof}
Suppose that $T \in \overline{\T}$, and let $W \subseteq V$ be a finite-dimensional $T$-invariant subspace. Since $T \in \overline{\T}$, there exists $S \in \T$ that agrees with $T$ on $W$. Since $S$ is triangularizable, by Theorem~\ref{tri-characterization}, there is a polynomial $p(x) \in k[x]\setminus k$ that factors into linear terms, such that $p(S)$ annihilates $W$. It follows that $p(T|_W)$ annihilates $W$ as well, and hence, by Theorem~\ref{cl-tri-thrm} (or Theorem~\ref{tri-characterization}), $T|_W$ is triangularizable. 

Conversely, suppose that the restriction of $T$ to any finite-dimensional $T$-invariant subspace of $V$ is triangularizable, and let $\U$ be an open neighborhood of $T$. Passing to a subset, if necessary, we may assume that $$\U = \{H \in \End_k(V) \mid H|_W = T|_W\}$$ for some finite-dimensional subspace $W$ of $V$. We shall show that $\U$ contains a triangularizable transformation, from which the desired conclusion follows.

We may view $V$ as a $k[x]$-module, where $x$ acts on $V$ as $T$. Then $M=k[x]W$ is a finitely-generated $k[x]$-submodule of $V$. Since $k[x]$ is a principal ideal domain, $M \cong k[x]^r \oplus N$, where $r \in \N$ and $N$ is a torsion $k[x]$-module. (See, e.g.,~\cite[Section 12.1, Theorem 5]{DF}.) Hence there exist subspaces $M_1, M_2 \subseteq V$ such that $M= M_1 \oplus M_2$, $M_1$ has a basis of the form $$\{T^i(v_j) \mid 1 \leq j \leq r, \, i \in \N\}$$ (where $T^{i_1}(v_{j_1}) \neq T^{i_2}(v_{j_2})$ whenever $(i_1,j_1) \neq (i_2,j_2)$), and for every $w \in M_2$ there is some $p(x) \in k[x]\setminus k$ such that $p(T)(w)=0$. In particular, every $w \in M_2$ is contained in a finite-dimensional $T$-invariant subspace of $V$. (Specifically, the space spanned by $w, T(w), T^2(w), \dots, T^{n-1}(w)$, where $n$ is the degree of a polynomial $p(x) \in k[x]\setminus k$ such that $p(T)(w)=0$, is invariant under $T$.)

Since $W$ is finite-dimensional, we can find finite-dimensional subspaces $W_1 \subseteq M_1$ and $W_2 \subseteq  M_2$ such that $W \subseteq W_1 \oplus W_2$. Since, by the above, $W_2$ is contained in a finite-dimensional $T$-invariant subspace of $V$, upon enlarging $W_2$, if necessary, we may assume that it is $T$-invariant. Hence, by hypothesis, $T|_{W_2}$ is triangularizable.

Upon enlarging $W_1$, if necessary, we may assume that $W_1$ has a basis of the form $$\{T^i(v_j) \mid 1\leq j \leq r, \, 0\leq i \leq n_j\},$$ for some $n_1, \dots, n_r \in \N$. Let $$W_1^+ = \langle \{T^i(v_j) \mid 1\leq j \leq r, \, 0\leq i \leq n_j+1\}\rangle.$$ Define $S \in \End_k(V)$ on $\{T^i(v_j) \mid 1\leq j \leq r, \, 0\leq i \leq n_j+1\}$ by
$$S (T^i(v_j)) = \left\{ \begin{array}{ll}
T^{i+1}(v_j) & \text{if } \, i \leq n_j\\
0 & \text{if } \, i=n_j+1
\end{array}\right.,$$
and extend $S$ to a transformation on $V$ by letting it act as $T$ on $W_2$ and as the zero transformation on a complement of $W_1^+ \oplus W_2$. Then $S$ agrees with $T$ on $W$, and hence $S \in \U$. Moreover, $S$ is triangularizable, since $T|_{W_2}$ is triangularizable, while $S|_{W_1^+}$ is nilpotent, and hence triangularizable, by Theorem~\ref{cl-tri-thrm} (or Lemma~\ref{upper-tri-lemma} or Theorem~\ref{tri-characterization}). 

For the final claim, suppose that $k$ is algebraically closed, and let $T \in \End_k(V)$. Suppose also that $W \subseteq V$ is a finite-dimensional $T$-invariant subspace. Then, by the Cayley-Hamilton theorem, $T$ satisfies a polynomial on $W$. Since $k$ is algebraically closed, this polynomial can be factored into linear terms in $k[x]$. Hence $T|_W$ is triangularizable, by Theorem~\ref{cl-tri-thrm}. It follows that $T \in \overline{\T}$, and hence $\overline{\T} = \End_k(V)$.
\end{proof}

Using the function topology and Theorem~\ref{tri-characterization} we can generalize the standard fact that a matrix is nilpotent if and only if it is similar to a strictly upper-triangular matrix if and only if $0$ is its only eigenvalue (over the algebraic closure of the base field).

\begin{proposition} \label{top-nilpt-prop}
Let $k$ be a field and $\, V$ a nonzero $k$-vector space. The following are equivalent for any $T \in \End_k(V)$.
\begin{enumerate}
\item[$(1)$] $T$ is \emph{topologically nilpotent} with respect to the function topology on $\, \End_k(V)$. That is, the sequence $\, (T^i)_{i=1}^\infty$ converges to $\, 0$.
\item[$(2)$] $V = \bigcup_{i=1}^\infty \ker (T^i)$.
\item[$(3)$] $T$ is strictly triangularizable.
\item[$(4)$] $T$ is triangularizable, and if $\, (\B, \leq)$ is a well-ordered basis for $\, V$ with respect to which $T$ is triangular, then $T$ is strictly triangular with respect to $\, (\B, \leq)$.
\item[$(5)$] $T$ is triangularizable, and $\, \ker(T-aI) \neq 0$ if and only if $a=0$, for all $a \in k$.
\end{enumerate}
\end{proposition}

\begin{proof}
We shall show that $(1) \Leftrightarrow (2) \Rightarrow (4) \Rightarrow (3) \Rightarrow (5) \Rightarrow (2)$.

$(1) \Leftrightarrow (2)$ $T$ is topologically nilpotent if and only if for every open neighborhood $\U$ of $0$ there exists $n \in \Z^+$ such that $T^n \in \U$. By our description of the function topology, this is equivalent to: for every finite-dimensional subspace $W$ of $V$ there exists $n \in \Z^+$ such that $T^n(W)=0$. That statement is clearly equivalent to $V = \bigcup_{i=1}^\infty \ker (T^i)$.  

$(2) \Rightarrow (4)$  If $T$ satisfies (2), then it is triangularizable, by Lemma~\ref{upper-tri-lemma}. Now let $(\B, \leq)$ be a well-ordered basis for $V$ with respect to which $T$ is triangular, and let $v \in \B$. Write $T(v) = a v + \sum_{u < v} a_u u$ for some $u \in \B$ and $a, a_u \in k$, and suppose that $a \neq 0$. Then for all $n\in \Z^+$ we have $T^n(v) = a^n v + w$ for some $w \in \langle \{u \in \B \mid u < v\}\rangle$, and hence $T^n(v) \neq 0$, producing a contradiction. Therefore $a = 0$, and hence $T(v) \in \langle \{u \in \B \mid u < v\}\rangle$ for all $v \in \B$. That is, $T$ is strictly triangular with respect to $(\B, \leq)$.

$(4) \Rightarrow (3)$ This is a tautology.

$(3) \Rightarrow (5)$ Suppose that $T$ is strictly triangular with respect to a well-ordered basis $(\B, \leq)$ for $V$. Then $T$ is triangularizable, by definition. Now, let $v \in V$, and write $v = a_u u + \sum_{w < u}a_w w$ for some $u, w \in \B$ and $a_u, a_w\in k$. Then $T$ being strictly triangular with respect to $(\B, \leq)$ implies that the coefficient of $u$ in $T(v)$, when expressed as a linear combination of elements of $\B$, is zero. Therefore, given $a \in k$, we can have $T(v) = av$ only if $a=0$. That is, $a=0$ whenever $\ker(T-aI) \neq 0$. On the other hand, since $T$ is triangularizable and $V\neq 0$, we have $\ker(T-aI) \neq 0$ for some $a \in k$, by Lemma~\ref{nontriv-invar-subspace}, from which (5) follows.

$(5) \Rightarrow (2)$ Suppose that $T$ satisfies (5), and let $v \in V$. By Theorem~\ref{tri-characterization}, there is a polynomial $p(x) \in k[x]\setminus k$ that factors into linear terms in $k[x]$, such that $v \in \ker (p(T))$. By Lemma~\ref{lemma:eigenspaces} and (5), this means that $p(x)$ can be taken to be $x^n$ for some $n \in \Z^+$, and therefore $T^n(v) = 0$. It follows that $V = \bigcup_{i=1}^\infty \ker (T^i)$.
\end{proof}

\section{Transformations Satisfying a Polynomial} \label{poly-section}

As we saw in Theorem~\ref{tri-characterization}, every triangularizable transformation $T$ satisfies a polynomial on each finite-dimensional $T$-invariant subspace. It is therefore natural to ask whether more can be said about transformations that satisfy a single polynomial on the entire space. That indeed can be quickly accomplished with the help of the following classical result from~\cite{Koethe}. (See also~\cite[Corollary 3.3]{BGMS} for a noncommutative generalization.)

\begin{theorem}[K\"{o}the]\label{Koethe-thrm}
Let $R$ be a commutative artinian ring. Then every $R$-module is a direct sum of cyclic $R$-modules if and only if $R$ is a principal ideal ring.
\end{theorem}

Applying K\"{o}the's theorem to the linear algebra setting yields the following extension of the rational canonical form to transformations of an arbitrary vector space that satisfy a polynomial. This was also observed by Radjabalipour in~\cite[Theorem 1.5]{Radjabalipour}, using a more elementary approach.

\begin{corollary} \label{poly-thrm}
Let $k$ be a field, $V$ a $k$-vector space, $T \in \End_k(V)$, and $p(x) \in k[x]\setminus k$ such that $p(T)=0$. Then $$V=\bigoplus_{\lambda \in \Lambda} \langle \{v_\lambda, T(v_\lambda), \dots, T^{n-1}(v_\lambda) \}\rangle$$ for some $v_\lambda \in V$, where $n$ is the degree of $p(x)$. 
\end{corollary}

\begin{proof}
Since $k[x]$ is a principal ideal domain, $R= k[x]/(p(x))$ is a (commutative) principal ideal ring (as its ideals correspond to the ideals of $k[x]$ containing $(p(x))$). Moreover, since $R$ is finite-dimensional as a $k$-vector space (being spanned by $\{1, x, x^2, \dots, x^{n-1}\}$), it is also artinian. Hence, by Theorem~\ref{Koethe-thrm}, every $R$-module is a direct sum of cyclic $R$-modules.

Now, viewing $V$ as an $R$-module, by letting $x$ act as $T$, we see that $V=\bigoplus_{\lambda \in \Lambda} Rv_\lambda$ for some $v_\lambda \in V$, from which the desired conclusion follows.
\end{proof}

The next definition will help us apply Corollary~\ref{poly-thrm} to triangularizable transformations, and thereby extend the Jordan canonical form to transformations of an arbitrary vector space that satisfy a polynomial.

\begin{definition}
Let $k$ be a field, $V$ a finite-dimensional $k$-vector space, and $T \in \End_k(V)$. If there is a basis $\, \{v_0, v_1, \dots, v_n\}$ for $\, V$ such that $T(v_i) = v_{i-1}$ for all $\, 1 \leq i \leq n$ and $T(v_0) = 0$, then we say that $T$ acts as a \emph{left shift transformation} on $\, V$.
\end{definition}

\begin{corollary}
Let $k$ be a field, $V$ a $k$-vector space, $p(x) \in k[x] \setminus k$ a polynomial that factors into linear terms in $k[x]$, and $T \in \End_k(V)$ such that $p(T) = 0$. Then there are finite-dimensional subspaces $\, V_\lambda \subseteq V$ and $a_\lambda \in k$ $(\lambda \in \Lambda)$, such that $\, V = \bigoplus_{\lambda \in \Lambda} V_\lambda$ and $T-a_\lambda I$ acts as a left shift transformation on $\, V_\lambda$, for each $\lambda \in \Lambda$.
\end{corollary}

\begin{proof}
By Corollary~\ref{poly-thrm}, $V$ can be written as a direct sum of finite-dimensional $T$-invariant subspaces. The desired conclusion now follows from applying Theorem~\ref{cl-tri-thrm} to each of these subspaces.
\end{proof}

\section{Double-Centralizer}

We conclude the paper by generalizing the following result for (finite) matrices to transformations of vector spaces of arbitrary dimension. See, e.g.,~\cite[Chapter 1, Theorem 7]{ST} or~\cite[Theorem 1]{PL} for proofs of this result.

\begin{theorem}[Classical Double-Centralizer Theorem] \label{theorem:classicaldoublecent}
Let $k$ be a field, $n \in \Z^+$, $\M_n(k)$ the $k$-algebra of all $n\times n$ matrices over $k$, and $T \in \M_n(k)$. Then $C(C(T)) = k[T]$.
\end{theorem}

We require a couple of standard lemmas.

\begin{lemma} \label{l.Ring}
If $R$ is a Hausdorff topological ring, then the centralizer of any subset of $R$ is closed in $R$.
\end{lemma}

\begin{proof}
Let $X$ be a subset of $R$, and let $\overline{C(X)}$ denote the closure of $C(X)$ in $R$. Suppose that $\overline{C(X)} \neq C(X)$. Then there must be some $r \in \overline{C(X)} \setminus C(X)$, and hence $rx - xr \neq 0$ for some $x \in X$. Since the topology is Hausdorff, there must be an open neighborhood $\U$ of $rx-xr$ such that $0 \notin \U$. By the continuity of the operations, we can find an open neighborhood $\mathcal{V}$ of $r$ such that $\mathcal{V}x - x \mathcal{V} \subseteq \U$. Since $C(X)$ is dense in $\overline{C(X)}$, there is some $r' \in C(X)$ such that $r' \in \mathcal{V}$. But then $$0=r'x-xr' \in \mathcal{V}x - x \mathcal{V} \subseteq \U,$$ contradicting $0 \notin \U$. Thus $\overline{C(X)} = C(X)$, i.e., $C(X)$ is closed.
\end{proof}

\begin{lemma} \label{lemma:doublecent}
Let $k$ be a field, let $\, V = W \oplus U$ be $k$-vector spaces, and let $T \in \End_k(V)$. If $\, W$ and $U$ are $T$-invariant, then $\, W$ and $U$ are also invariant under every element of $C(C(T))$.
\end{lemma}

\begin{proof}
Let $\pi \in \End_k(V)$ be the projection of $V$ onto $W$ with kernel $U$, and let $S \in C(C(T))$. Since $W$ and $U$ are $T$-invariant, we have $\pi \in C(T)$, and hence $S \pi = \pi S$. Thus $$S(W) = S\pi (W) = \pi S(W) \subseteq \pi(V) = W,$$ and similarly $S(U) \subseteq U$.
\end{proof}

\begin{proposition} \label{double-cent-prop}
Let $k$ be a field, $V$ a $k$-vector space, and $T \in \End_k(V)$. Suppose that there are finite-dimensional $T$-invariant subspaces $\, V_\lambda \subseteq V$ $(\lambda \in \Lambda)$ such that $\, V = \bigoplus_{\lambda \in \Lambda} V_\lambda$. Then $C(C(T))=\overline{k[T]}$.
\end{proposition}

\begin{proof}
Clearly $k[T] \subseteq C(C(T))$. Since, by Lemma~\ref{l.Ring}, $C(C(T))$ is closed, it follows that $\overline{k[T]} \subseteq C(C(T))$.

For the opposite inclusion, let $S \in C(C(T))$, and let $\U$ be an open neighborhood of $S$. Passing to a subset, if necessary, we may assume that $$\U = \{F \in \End_k(V) \mid F|_U = S|_U\}$$ for some finite-dimensional subspace $U$ of $V$. We can find some $\lambda_1, \dots, \lambda_m \in \Lambda$ such that $U \subseteq V_{\lambda_1} \oplus \dots \oplus V_{\lambda_m}$. Letting $W=V_{\lambda_1} \oplus \dots \oplus V_{\lambda_m}$, we have $S(W) \subseteq W$, by Lemma~\ref{lemma:doublecent}. 

Let $H' \in \End_k(W)$ be such that $H'T|_W=T|_WH'$. Extending $H'$ to a map $H \in \End_k(V)$ by letting $H(\bigoplus_{\lambda \in \Lambda \setminus \{\lambda_1, \dots, \lambda_m\}} V_\lambda)=0$, we see that $HT=TH$, and hence also $SH=HS$. Since $S(W) \subseteq W$, we have $S|_WH'=H'S|_W$. Since $H' \in C(T|_W)$ was arbitrary, this shows that $S|_W \in C(C(T|_W))$. Thus, by Theorem~\ref{theorem:classicaldoublecent}, we have $S|_W \in k[T|_W]$. Since $W$ is $T$-invariant, this implies that there is some polynomial $p(x) \in k[x]$ such that $S|_W = p(T)|_W$. Therefore $p(T) \in \mathcal{U}$, and hence $S$ is a limit point of $k[T]$. It follows that $S \in \overline{k[T]}$, and thus $C(C(T))\subseteq \overline{k[T]}$.
\end{proof}

With the help of the next lemma, we can give another generalization of Theorem~\ref{theorem:classicaldoublecent} to infinite-dimensional vector spaces.

\begin{lemma} \label{poly-cl-lemma}
Let $k$ be a field, $V$ a $k$-vector space, $p(x) \in k[x]\setminus k$, and $T \in \End_k(V)$ such that $p(T)=0$. Then $k[T] = \overline{k[T]}$.
\end{lemma}

\begin{proof}
Clearly $k[T] \subseteq \overline{k[T]}$. To show the opposite inclusion, let us take $S \in \overline{k[T]}$, and prove that $S \in k[T]$.

By the properties of the function topology, for each finite-dimensional subspace $W \subseteq V$ there exists $q(x) \in k[x]$ such that $S|_W = q(T)|_W$. For each $W$ let $q_W(x) \in k[x]$ be such a polynomial of least degree. Since $p(T)=0$, we have $\deg(q_W) < \deg (p)$ for each $W$, by the division algorithm. Thus, we can find a finite-dimensional subspace $U \subseteq V$ such that $\deg(q_U) \geq \deg(q_W)$ for all $W$.

Now let $W \subseteq V$ be any finite-dimensional subspace. Then $q_{U+W}(T)|_U = S|_U = q_U(T)|_U,$ and hence $\deg(q_{U+W}) = \deg(q_U)$, by our definition of the $q_W$ and choice of $U$. Thus, $q(x) = q_{U}(x)-q_{U+W}(x)$ is a polynomial of degree at most $\deg (q_U)$ such that $q(T)|_U = 0$. If $q(x)$ were nonzero, then this would imply, upon applying the division algorithm to $q_U(x)$ and $q(x)$, that there is a polynomial $q'(x) \in k[x]$ such that $\deg(q') <\deg(q) \leq\deg(q_U)$ and $S|_U = q'(T)|_U$, contradicting the minimality of the degree of $q_U$. Therefore $q(x) = 0$, and hence  $q_{U+W}(x) = q_U(x)$, which implies that $S|_W = q_{U+W}(T)|_W = q_U(T)|_W.$ Since $W$ was arbitrary, this means that $S=q_U(T)$, and hence $S \in k[T]$.
\end{proof}

\begin{theorem} \label{double-cent-thrm}
Let $k$ be a field, $V$ a $k$-vector space, and $T \in \End_k(V)$. If there exists $p(x) \in k[x] \setminus k$ such that $p(T)=0$, then $C(C(T))=\overline{k[T]}= k[T]$.
\end{theorem}

\begin{proof}
By Corollary~\ref{poly-thrm}, if $T$ satisfies the above condition, then it also satisfies the hypotheses of Proposition~\ref{double-cent-prop}, and hence $C(C(T))=\overline{k[T]}$. The desired conclusion now follows from Lemma~\ref{poly-cl-lemma}.
\end{proof}

The next example shows that the conclusion of Theorem~\ref{double-cent-thrm} does not hold for chain-triangularizable transformations.

\begin{example}
Let $k$ be a field and $V$ a $k$-vector space with basis $\{v_i \mid i \in \Z\}$. Define $T \in \End_k(V)$ by $T(v_i) = v_{i-1}$ for each $i \in \Z$, and extend linearly to all of $V$. As seen in Example~\ref{weak-tri-eg}, $T$ is chain-triangularizable but not triangularizable. We shall show that $$C(C(T)) = C(T) = k[T,T^{-1}] = \overline{k[T,T^{-1}]} \neq \overline{k[T]}.$$

Since $T$ is clearly invertible, $k[T,T^{-1}] \subseteq C(T)$, and since, by Lemma~\ref{l.Ring}, $C(T)$ is closed, we have $\overline{k[T,T^{-1}]} \subseteq C(T)$. Now let $S \in C(T)$ be any element, and write $S(v_0) = \sum_{i=-n}^m a_iv_i$ for some $n,m \in \N$ and $a_i \in k$. Then for every $l \in \Z$, we have $$S(v_l) = ST^{-l}(v_0) = T^{-l}S(v_0) = \sum_{i=-n}^m a_iT^{-l}(v_i) = \sum_{i=-n}^m a_iv_{i+l} = \sum_{i=-n}^m a_iT^{-i}(v_{l}),$$ and therefore $S = \sum_{i=-n}^m a_iT^{-i} \in k[T,T^{-1}].$ Thus $C(T) \subseteq k[T,T^{-1}]$. Combining this with $\overline{k[T,T^{-1}]} \subseteq C(T)$, we conclude that $$C(T) = k[T,T^{-1}] = \overline{k[T,T^{-1}]}.$$ Since $C(C(T))$ is the center of the ring $C(T)$, and $C(T) = k[T,T^{-1}]$ is commutative, we also have $C(T) = C(C(T))$.

It remains to show that $T^{-1} \notin \overline{k[T]}$, from which we can conclude that $\overline{k[T,T^{-1}]} \neq \overline{k[T]}$. Suppose, on the contrary, that $T^{-1} \in \overline{k[T]}$. Then there must be some $S \in k[T]$ such that $S(v_0) = v_{1} = T^{-1}(v_0)$. But $S = \sum_{i=0}^m a_iT^i$ for some $m \in \N$ and $a_i \in k$, and hence $$S(v_0) = \sum_{i=0}^m a_iT^i(v_0) = \sum_{i=0}^m a_iv_{-i} \neq v_1,$$ producing a contradiction. \hfill $\Box$
\end{example}

The previous theorem and example leave us with the following question.

\begin{question}
Let $k$ be a field, $V$ a $k$-vector space, and $T \in \End_k(V)$ a triangularizable transformation $($or, more generally, a transformation such that every finite-dimensional subspace of $\, V$ is annihilated by $p(T)$ for some $p(x) \in k[x]$$)$. Is it the case that $C(C(T))=\overline{k[T]}$?
\end{question}

\section*{Acknowledgements}

I am grateful to George Bergman for his numerous comments on an earlier version of this paper, which have led to significant improvements, including the addition of condition (5) to Theorem~\ref{tri-characterization}. I also would like to thank the referee for a very thoughtful review, and particularly for the suggestion to add condition (7) to Theorem~\ref{tri-characterization}. Finally, I would like to thank Greg Oman for helpful conversations about this material.

\vspace{.1in}

\noindent
Department of Mathematics, University of Colorado, Colorado Springs, CO, 80918, USA \newline
\noindent {\href{mailto:zmesyan@uccs.edu}{zmesyan@uccs.edu}}


\begin{thebibliography}{00}
\bibitem{BGMS}
M.\ Behboodi, A.\ Ghorbani, A.\ Moradzadeh-Dehkordi, and S.\ H.\ Shojaee, \emph{On Left K\"{o}the Rings and a Generalization of the K\"{o}the-Cohen-Kaplansky Theorem,} Proc.\ Amer.\ Math.\ Soc.\ \textbf{142} (2014) 2625--2631.

\bibitem{DF}  D.\ S.\ Dummit and R.\ M.\ Foote, \emph{Abstract Algebra, 3rd Edition,} John Wiley and Sons, Inc., Hoboken, New Jersey, 2004.

\bibitem{IMR}
M.\ C.\ Iovanov, Z.\ Mesyan, and M.\ L.\ Reyes, \emph{Infinite-Dimensional Diagonalization and Semisimplicity,} Israel J.\ Math.\ \textbf{215} (2016) 801--855.

\bibitem{Koethe}
G.\ K\"{o}the, \emph{Verallgemeinerte Abelsche Gruppen mit hyperkomplexem Operatorenring,} Math.\ Z.\ \textbf{39} (1935) 31--44.

\bibitem{PL}
P.\ Lagerstrom, \emph{A Proof of a Theorem on Commutative Matrices,} Bull.\ Amer.\ Math.\ Soc.\ \textbf{51} (1945) 535--536.

\bibitem{Radjabalipour}
M.\ Radjabalipour, \emph{Infinite-Dimensional Versions of the Primary, Cyclic and Jordan Decompositions,} Bull.\ Iranian Math.\ Soc.\ \textbf{41} (2015) 175--183.

\bibitem{RR}
H.\ Radjavi and P.\ Rosenthal, \emph{Simultaneous Diagonalization,} Springer-Verlag, New York, 2000.

\bibitem{ST} 
D.\ A.\ Suprunenko and R.\ I.\ Tyshkevich, \emph{Commutative Matrices,} Academic Press, New York, 1968. 

\bibitem{Warner}
S.\ Warner, \emph{Topological Rings,} North-Holland Mathematical Studies, 178, North-Holland, Amsterdam, 1993.

\end{thebibliography}
\end{document}